\documentclass[11pt]{amsart}

\usepackage[margin=1.15in]{geometry}
\usepackage{amsmath,amssymb,amsthm,mathtools}
\usepackage{mathrsfs}
\usepackage{booktabs,array}
\usepackage{microtype}
\usepackage{enumitem}
\usepackage[hidelinks]{hyperref}
\usepackage[nameinlink,capitalise]{cleveref}
\usepackage{xcolor}

\numberwithin{equation}{section}

\newtheorem{theorem}{Theorem}[section]
\newtheorem{proposition}[theorem]{Proposition}
\newtheorem{lemma}[theorem]{Lemma}
\newtheorem{corollary}[theorem]{Corollary}
\theoremstyle{definition}
\newtheorem{definition}[theorem]{Definition}
\theoremstyle{remark}
\newtheorem{remark}[theorem]{Remark}
\newcommand{\Li}{\operatorname{Li}}
\newcommand{\F}{\mathcal F}
\newcommand{\R}{\mathcal R}
\newcommand{\PP}{\mathbb P}
\newcommand{\CC}{\mathbb C}
\newcommand{\ZZ}{\mathbb Z}

\newcommand{\loc}{\mathrm{loc}}
\newcommand{\con}{\mathrm{con}}

\newcommand{\lcm}{\operatorname{lcm}}

\newcommand{\ord}{\operatorname{ord}}

\title[Difference Equations and GV Reconstruction]{Difference Equations and Reconstruction of Local Gopakumar-Vafa Spectra for Threefold Flops}
\author{Xiaobin Li}
\address{School of Mathematics, Southwest Jiaotong University,
West zone, High-tech district, Chengdu, Sichuan 611756, China}
\date{September 14, 2026}

\begin{document}
\begin{abstract}
We study difference equations for the local Gromov-Witten theory of a contractible rational curve in a Calabi-Yau threefold from the viewpoint of reconstruction. Starting from the Bryan-Katz-Leung multiple cover decomposition, we associate to all nonzero local GV degree sectors on an exceptional ray a single central difference forcing term governed by explicit Fej\'er-type Laurent kernels.

Our main result is an information equivalence theorem: once the exceptional ray and its lattice normalization are fixed, the full quantum forcing, its regular specialization at $\lambda=0$, and the local genus zero Gopakumar-Vafa spectrum determine one another. In particular, every local GV multiplicity is recovered from the classical forcing by an explicit M\"obius inversion. In the primitive normalization relevant to a simple flopping contraction, the same classical forcing also recovers the noncommutative and commutative widths of the contraction algebra: the former is the order at $q=1$ of the exponentiated forcing and the latter is its first $q$-coefficient.

For finite support we further determine the minimal Laurent polynomial denominator clearing operator on the finest common shift lattice. Its cyclotomic factorization is strictly coarser than the GV spectrum. Length two, $E_6$, $E_7$, and $E_8$ examples exhibit the resulting hierarchy of information loss: Koll\'ar length does not determine the forcing, the cyclotomic skeleton does not determine the GV support, and even the complete local difference data do not determine the analytic type of the flop.
\end{abstract}

\keywords{Gromov-Witten invariants, Gopakumar-Vafa invariants, threefold flops, difference equations, contraction algebras}

\maketitle

\section{Introduction and main results}

Difference equations provide a compact organization of the all genus Gromov-Witten theory of the resolved conifold. Let
\[
q=e^{2\pi i t},\qquad h=\frac{\lambda}{2\pi},\qquad D=\frac1{2\pi}\frac{\partial}{\partial t}.
\]
The nonconstant resolved conifold potential is
\[
\Phi_1(\lambda,t)=\sum_{k\ge1}\frac{q^k}{k(2\sin(k\lambda/2))^2}.
\]
Alim proved that it satisfies the central difference equation \cite{Alim2023}
\begin{equation}\label{eq:alim}
\Phi_1(t+h)-2\Phi_1(t)+\Phi_1(t-h)=D^2\Li_3(q).
\end{equation}
He also observed degreewise difference equations for the genus zero GV-resummed contribution of a general Calabi-Yau threefold. More recently, Alim and Tischler have studied a complementary analytic direction for the resolved conifold, in which the resurgent completion of the closed Gromov-Witten potential has a strong coupling expansion governed in part by relative Gromov-Witten invariants \cite{AlimTischler2026}. These developments suggest that difference and $q$-difference structures retain geometric information beyond their role as compact functional identities.

There is also a substantial, largely parallel tradition of difference equations in quantum $K$-theory. Givental and Lee related the quantum $K$-theory of flag manifolds to finite-difference Toda lattices \cite{GiventalLee2003}, while Lee and Pandharipande established a reconstruction theorem in genus zero quantum cohomology and quantum $K$-theory \cite{LeePandharipande2004}. Iritani, Milanov and Tonita later reconstructed big quantum $K$-rings from the $q$-difference module structure \cite{IritaniMilanovTonita2015}, and Ruan and Wen developed the $q$-difference viewpoint through explicit calculations and convergence questions \cite{RuanWen2022}. In the Calabi-Yau threefold setting, Garoufalidis and Scheidegger formulated a conjectural small quantum $K$-theoretic $q$-difference equation for the quintic whose coefficients are expressed linearly in terms of Gopakumar-Vafa invariants \cite{GaroufalidisScheidegger2022}, and Wen studied connection and confluence problems for the $K$-theoretic $q$-difference equation of the quintic \cite{Wen2022}. These quantum $K$-theoretic equations act on $I$- or $J$-functions through multiplicative shifts of Novikov variables. The equations studied here are of a different nature: they act directly on the GV-resummed cohomological Gromov-Witten potential through additive shifts of the K\"ahler parameter determined by the topological string coupling. Correspondingly, our reconstruction problem is raywise and enumerative: it recovers the local genus zero GV multiplicities from the classical specialization of the forcing term, without invoking a quantum $K$-theory difference module structure.

The question addressed here is deliberately local and reconstruction theoretic:
\begin{quote}
\emph{Once a local flopping contribution has been organized by a difference equation, exactly which enumerative and geometric data can be recovered from its forcing term?}
\end{quote}
The answer is sharper than the existence of the difference equation itself. For a fixed exceptional ray, the full $\lambda$-dependent forcing and its classical specialization contain precisely the same genus zero GV information. Moreover, in primitive normalization the classical forcing directly records two numerical invariants of the contraction algebra. Passing further to the minimal denominator clearing operator loses information, and explicit flops realize each loss.

Let $Y$ be a smooth Calabi-Yau threefold and let $C\simeq\PP^1\subset Y$ be a smoothly embedded contractible rational curve. The multiple cover theory of Bryan-Katz-Leung associates to $C$ finitely many scheme theoretic multiplicities in degrees $1\le d\le \ell$, where $1\le \ell\le6$ is Koll\'ar's length \cite{BKL2001,KatzMorrison1992}. In modern GV language, these multiplicities form the local genus zero spectrum $\{n_d\}$. The nonconstant local Gromov-Witten contribution is
\begin{equation}\label{eq:localdecomp-intro}
\F_C^{\loc}=\sum_{d\in S}n_d\Phi_d,
\end{equation}
where
\[
S=\{d\ge1:n_d\ne0\},\qquad
\Phi_d(\lambda,t)=\sum_{k\ge1}\frac{q^{dk}}{k(2\sin(k\lambda/2))^2}.
\]
This finite sector decomposition is the geometric input for the paper. Put
\begin{equation}\label{eq:gdef-intro}
g=\gcd S,
\end{equation}
and define
\[
T_af(t)=f(t+a),\qquad \Delta_g=T_{h/g}-2+T_{-h/g}.
\]
For $m\ge1$ let
\begin{equation}\label{eq:Km-intro}
K_m(z)=\sum_{j=-(m-1)}^{m-1}(m-|j|)z^j
=z^{-(m-1)}(1+z+\cdots+z^{m-1})^2.
\end{equation}
Then
\[
K_m(e^{ix})=\left(\frac{\sin(mx/2)}{\sin(x/2)}\right)^2.
\]
The kernels $K_m$ measure the discrepancy between an individual degree sector and the common gcd lattice.

\subsection*{Main result I: reconstruction and information equivalence}

Set
\[
\R_g(\lambda,q):=\Delta_g\F_C^{\loc},\qquad R_g(q):=\R_g(0,q).
\]
Writing $a_e:=n_{ge}$, the classical forcing is
\begin{equation}\label{eq:intro-classical}
R_g(q)=-\sum_{e\ge1}e^2a_e\Li_1(q^{ge}).
\end{equation}
The divisor convolution in \eqref{eq:intro-classical} can be inverted explicitly.

\begin{theorem}[Arithmetic reconstruction and information equivalence]\label{thm:intro-recon}
Fix the exceptional ray and its lattice normalization $g$. Then the classical forcing series $R_g(q)$ uniquely determines the local genus zero GV spectrum. More precisely,
\begin{equation}\label{eq:intro-moebius}
\boxed{
n_{gN}=-\frac1{N^2}\sum_{r\mid N}\frac{\mu(r)}r
[q^{gN/r}]R_g(q).}
\end{equation}
Consequently the following three data sets mutually determine one another:
\begin{equation}\label{eq:intro-equivalence}
\boxed{
\R_g(\lambda,q)
\Longleftrightarrow
R_g(q)
\Longleftrightarrow
\{n_{gN}\}_{N\ge1}.}
\end{equation}
\end{theorem}

Thus the limit $\lambda\to0$ forgets the nontrivial shift profile of the quantum kernels but not the local genus zero GV spectrum. In this precise sense the classical forcing is already a complete coordinate system for the genus zero local curve counting data on the chosen exceptional ray.

\subsection*{Main result II: the common difference equation}

The functional mechanism behind the reconstruction theorem is a single difference operator acting simultaneously on every nonzero degree sector.

\begin{theorem}[Common difference equation]\label{thm:intro-common}
Let $\F_C^{\loc}$ be as above. Then
\begin{equation}\label{eq:intro-common1}
\boxed{
\Delta_g\F_C^{\loc}
=
D^2\sum_{d\in S}\frac{n_d}{d^2}
K_{d/g}(T_{h/d})\Li_3(q^d).}
\end{equation}
Equivalently,
\begin{equation}\label{eq:intro-common2}
\boxed{
\Delta_g\F_C^{\loc}
=-\sum_{d\in S}n_dK_{d/g}(T_{h/d})\Li_1(q^d).}
\end{equation}
\end{theorem}

For $S=\{1\}$ this reduces to \eqref{eq:alim}. The point is not merely that every degree has its own equation: all sectors are placed on a single lattice determined by $g$, while their arithmetic mismatch is encoded by the finite Fej\'er kernels $K_{d/g}$.

\subsection*{Main result III: contraction algebra widths from the forcing}

For a simple flopping contraction in primitive curve class normalization, write $g=1$ and set
\[
\mathscr E_C(q):=\exp(R_1(q)).
\]
Using $\Li_1(z)=-\log(1-z)$ gives the finite product
\begin{equation}\label{eq:intro-E-product}
\mathscr E_C(q)=\prod_{d\in S}(1-q^d)^{d^2n_d}.
\end{equation}
Toda's formula expresses the noncommutative and commutative widths of the flopping curve in terms of the same GV spectrum \cite{Toda2015}.

\begin{proposition}[Contraction algebra bridge]\label{prop:intro-contraction}
Let $A_{\mathrm{con}}$ be the contraction algebra of a one-curve flopping contraction. In primitive normalization,
\begin{equation}\label{eq:intro-widths}
\boxed{
\ord_{q=1}\mathscr E_C(q)=\dim_{\CC}A_{\mathrm{con}},
\qquad
-[q]R_1(q)=\dim_{\CC}A_{\mathrm{con}}^{\mathrm{ab}}.}
\end{equation}
In particular, the classical difference forcing determines both numerical widths of the contraction algebra.
\end{proposition}

This connects the difference equation package directly to noncommutative flop geometry: the forcing does not recover the contraction algebra itself, but it recovers two of its fundamental numerical invariants.

\subsection*{Cyclotomic compression and information loss}

The gcd lattice is adapted to the common central difference equation. For simultaneous denominator clearing it is more natural to use the finest common shift lattice
\[
L=\lcm(S),\qquad Z=T_{h/L}.
\]
A Laurent polynomial $P(Z)\in\CC[Z^{\pm1}]$ is a denominator clearing operator if
\[
(Z^{L/d}-1)^2\mid P(Z)\qquad\text{for every }d\in S.
\]
Minimality is with respect to divisibility modulo Laurent units.

\begin{theorem}[Cyclotomic minimality]\label{thm:intro-cyclo}
For a finite nonempty support $S$, the unique minimal denominator clearing operator, up to a Laurent unit, is
\begin{equation}\label{eq:intro-pmin}
\boxed{
P_S^{\min}(Z)=\lcm_{d\in S}(Z^{L/d}-1)^2.}
\end{equation}
If
\[
\mathscr C(S)=\bigcup_{d\in S}\{r\ge1:r\mid L/d\},
\]
then
\begin{equation}\label{eq:intro-cyclofac}
\boxed{
P_S^{\min}(Z)=\prod_{r\in\mathscr C(S)}\Phi_r(Z)^2}
\end{equation}
up to a Laurent unit. If $1\in S$, then $P_S^{\min}(Z)=(Z^L-1)^2$ up to a Laurent unit, with symmetric representative $\Delta_1=T_h-2+T_{-h}$.
\end{theorem}

The reconstruction theorem and the cyclotomic theorem therefore describe opposite ends of the same hierarchy: the forcing retains the complete genus zero GV spectrum, while the minimal denominator clearing operator is a compressed shadow of it.

\subsection*{Geometric examples}

We apply the theory to the resolved conifold, the standard Laufer flop, and higher length $E_6$, $E_7$, and $E_8$ constructions \cite{CollinucciEtAl2022}. For the length two spectrum $(5,1)$,
\begin{equation}\label{eq:intro-laufer}
\Delta_1\F_L^{\loc}
=
\log\!\left[(1-q)^5(1-q^2)^2(1-q^2e^{i\lambda})(1-q^2e^{-i\lambda})\right].
\end{equation}
The higher length models realize successively the kernels $K_3,K_4,K_5,K_6$ and exhibit three distinct failures of reconstruction after one passes to coarser data:
\begin{equation}\label{eq:intro-hierarchy}
\boxed{
\R_g\Longleftrightarrow R_g\Longleftrightarrow\{n_d\}
\Longrightarrow S\Longrightarrow P_S^{\min}.}
\end{equation}
The first three objects carry the same genus zero enumerative information, but both subsequent arrows are genuinely lossy. The length five and length six $E_8$ examples have different supports and the same cyclotomic skeleton. At the opposite end, Brown and Wemyss constructed analytically non-isomorphic flops with identical GV invariants \cite{BrownWemyss2018}; by the reconstruction theorem they therefore have identical full and classical forcing terms. Thus
\begin{equation}\label{eq:intro-fullhierarchy}
\boxed{
\text{analytic flop geometry}
\longrightarrow
\bigl(\R_g\Longleftrightarrow R_g\Longleftrightarrow\{n_d\}\bigr)
\longrightarrow S\longrightarrow P_S^{\min},}
\end{equation}
and the explicit examples in Section~\ref{sec:examples} show that the reverse implications fail at the indicated stages.

The paper is organized as follows: Section~\ref{sec:local} derives the common gcd lattice difference equation from the local multiple cover decomposition. Section~\ref{sec:arithmetic} proves M\"obius and quantum reconstruction, establishes information equivalence, relates the classical forcing to contraction algebra widths, and determines the minimal cyclotomic denominator clearing operator. Section~\ref{sec:examples} treats explicit threefold flops and realizes the information losses geometrically. A short final section places these results in the broader difference equation and local curve counting picture.

\section{Local Gromov-Witten potentials and difference kernels}\label{sec:local}

\subsection{Local multiple cover contributions}

Let $Y$ be a smooth Calabi-Yau threefold and let $C\simeq\PP^1\subset Y$ be a smoothly embedded contractible rational curve. Bryan-Katz-Leung associate to $C$ a finite sequence of Hilbert scheme multiplicities $k_1,\dots,k_\ell$, where $1\le\ell\le6$, and prove the corresponding arbitrary genus local multiple cover formula \cite{BKL2001}. In the genus zero GV interpretation one writes $n_d=k_d$; see also \cite{Katz2008,Toda2015}.

\begin{proposition}\label{prop:decomp}
The nonconstant local Gromov-Witten contribution associated with $C$ has the form
\begin{equation}\label{eq:decomp}
\boxed{
\F_C^{\loc}(\lambda,t)=\sum_{d=1}^{\ell}n_d\Phi_d(\lambda,t),}
\end{equation}
where
\begin{equation}\label{eq:Phid}
\boxed{
\Phi_d(\lambda,t)=\sum_{k\ge1}\frac{q^{dk}}{k(2\sin(k\lambda/2))^2}.}
\end{equation}
Equivalently, if $S=\{d:n_d\ne0\}$, then $\F_C^{\loc}=\sum_{d\in S}n_d\Phi_d$.
\end{proposition}

\begin{remark}
The role of \cref{prop:decomp} is geometric: it supplies the finite collection of local degree sectors. The purpose of the difference formalism is to solve two global problems across these sectors simultaneously---to construct one common forcing equation and to determine exactly which local enumerative data can be reconstructed from it. We deliberately refer to $\F_C^{\loc}$ as the \emph{nonconstant local contribution}; constant map terms and global compactification data play no role in the present discussion.
\end{remark}

\subsection{Shift operators}

Set $h=\lambda/(2\pi)$ and define $T_af(t)=f(t+a)$. Since $q=e^{2\pi it}$,
\begin{equation}\label{eq:shiftmode}
T_aq^N=e^{2\pi iNa}q^N.
\end{equation}
For each $d\ge1$, define the elementary central difference
\begin{equation}\label{eq:Deltad}
\Delta_d=T_{h/d}-2+T_{-h/d}.
\end{equation}
On $q^{dk}$,
\begin{equation}\label{eq:Deltadmode}
\Delta_dq^{dk}=-4\sin^2(k\lambda/2)q^{dk}.
\end{equation}
Consequently
\begin{equation}\label{eq:sector-basic}
\Delta_d\Phi_d=-\Li_1(q^d).
\end{equation}
Moreover, since $Dq^N=iNq^N$,
\begin{equation}\label{eq:D2Li3}
D^2\Li_3(q^d)=-d^2\Li_1(q^d),
\end{equation}
and therefore
\begin{equation}\label{eq:sector-D2}
\boxed{
\Delta_d\Phi_d=\frac1{d^2}D^2\Li_3(q^d).}
\end{equation}
Equation \eqref{eq:sector-D2} is the degree $d$ version of the resolved conifold identity. The problem is to replace the family $\{\Delta_d\}$ by one operator acting simultaneously on every nonzero sector.

\subsection{Multiple cover kernels}

Let $S=\{d:n_d\ne0\}$ and $g=\gcd S$. Every $d\in S$ can be written uniquely as $d=gm$ with $m\ge1$. Define
\[
\Delta_g=T_{h/g}-2+T_{-h/g}.
\]
For $m\ge1$ define
\begin{equation}\label{eq:Km}
K_m(z)=\sum_{j=-(m-1)}^{m-1}(m-|j|)z^j.
\end{equation}

\begin{lemma}\label{lem:kernel}
For every $m\ge1$,
\begin{equation}\label{eq:Kfactor}
\boxed{
K_m(z)=z^{-(m-1)}(1+z+\cdots+z^{m-1})^2.}
\end{equation}
In particular,
\begin{equation}\label{eq:Ktrig}
\boxed{
K_m(e^{ix})=\left(\frac{\sin(mx/2)}{\sin(x/2)}\right)^2.}
\end{equation}
\end{lemma}

\begin{proof}
The coefficient of $z^r$ in $(1+z+\cdots+z^{m-1})^2$ is triangular. After multiplication by $z^{-(m-1)}$, the coefficient of $z^j$, $-(m-1)\le j\le m-1$, is $m-|j|$, proving \eqref{eq:Kfactor}. For $z=e^{ix}$,
\[
1+e^{ix}+\cdots+e^{i(m-1)x}=e^{i(m-1)x/2}\frac{\sin(mx/2)}{\sin(x/2)},
\]
and \eqref{eq:Ktrig} follows.
\end{proof}

The first kernels are
\[
K_1(z)=1,
\qquad
K_2(z)=z+2+z^{-1},
\]
\[
K_3(z)=z^2+2z+3+2z^{-1}+z^{-2}.
\]

\subsection{Sector cancellation on the gcd lattice}

The following identity is the basic computational mechanism of the paper.

\begin{lemma}[Sector cancellation]\label{lem:sector}
Let $d=gm$. Then
\begin{equation}\label{eq:sector-kernel}
\boxed{
\Delta_g\Phi_d=-K_m(T_{h/d})\Li_1(q^d).}
\end{equation}
Equivalently,
\begin{equation}\label{eq:sector-kernel-D2}
\boxed{
\Delta_g\Phi_d=\frac1{d^2}D^2K_m(T_{h/d})\Li_3(q^d).}
\end{equation}
\end{lemma}

\begin{proof}
Since $d=gm$,
\[
T_{h/g}q^{dk}=e^{imk\lambda}q^{dk},
\]
so
\begin{equation}\label{eq:Delta-g-mode}
\Delta_gq^{dk}=-4\sin^2(mk\lambda/2)q^{dk}.
\end{equation}
Applying this to \eqref{eq:Phid} gives
\[
\begin{aligned}
\Delta_g\Phi_d
&=-\sum_{k\ge1}\frac{q^{dk}}k
\frac{\sin^2(mk\lambda/2)}{\sin^2(k\lambda/2)}\\
&=-\sum_{k\ge1}\frac{q^{dk}}kK_m(e^{ik\lambda}),
\end{aligned}
\]
by \cref{lem:kernel}. On the other hand,
\[
T_{h/d}q^{dk}=e^{ik\lambda}q^{dk},
\]
so
\[
K_m(T_{h/d})\Li_1(q^d)=\sum_{k\ge1}\frac{q^{dk}}kK_m(e^{ik\lambda}).
\]
This proves \eqref{eq:sector-kernel}. Equation \eqref{eq:sector-kernel-D2} follows from \eqref{eq:D2Li3} and the commutativity of $D$ with shifts.
\end{proof}

\subsection{The common difference theorem}

\begin{theorem}[Common gcd lattice equation]\label{thm:common}
Let $C$ satisfy the hypotheses of \cref{prop:decomp}, let $S=\{d:n_d\ne0\}$, and let $g=\gcd S$. Then
\begin{equation}\label{eq:common-log}
\boxed{
\Delta_g\F_C^{\loc}
=-\sum_{d\in S}n_dK_{d/g}(T_{h/d})\Li_1(q^d).}
\end{equation}
Equivalently,
\begin{equation}\label{eq:common-D2}
\boxed{
\Delta_g\F_C^{\loc}
=D^2\sum_{d\in S}\frac{n_d}{d^2}K_{d/g}(T_{h/d})\Li_3(q^d).}
\end{equation}
\end{theorem}

\begin{proof}
By \cref{prop:decomp}, $\F_C^{\loc}=\sum_{d\in S}n_d\Phi_d$. Apply $\Delta_g$ and use \cref{lem:sector} sector by sector.
\end{proof}

\begin{remark}
The theorem should be distinguished from a collection of degreewise equations. Each $\Phi_d$ is naturally adapted to $\Delta_d$. Theorem~\ref{thm:common} places all nonzero degrees on the single lattice determined by $g=\gcd S$; the mismatch is measured exactly by the kernels $K_{d/g}$.
\end{remark}

\begin{remark}
For $S=\{1\}$ and $n_1=1$, one has $g=1$ and $K_1=1$, so \eqref{eq:common-D2} reduces to the resolved conifold equation \eqref{eq:alim}.
\end{remark}

\subsection{Formal interpretation}

The individual sector $\Phi_d$ is naturally an element of $\CC((\lambda))[[q]]$. After applying $\Delta_g$, the sine square poles cancel. By \eqref{eq:common-log}, the forcing term is a finite sum of expressions
\[
K_m(T_{h/d})\Li_1(q^d),
\]
which are regular at $\lambda=0$. Since $K_m(z)=K_m(z^{-1})$, the resulting Taylor expansion is even in $\lambda$. Hence
\[
\R_g(\lambda,q):=\Delta_g\F_C^{\loc}
\]
admits the specialization
\[
R_g(q):=\R_g(0,q),
\]
which is the starting point of the arithmetic reconstruction theory.

\section{Arithmetic reconstruction and cyclotomic minimality}\label{sec:arithmetic}

\subsection{The classical forcing series}

Write $d=ge$ and set $a_e:=n_{ge}$. By \eqref{eq:common-log},
\begin{equation}\label{eq:Rquantum}
\R_g(\lambda,q)=-\sum_{e\ge1}a_eK_e(T_{h/(ge)})\Li_1(q^{ge}),
\end{equation}
where only finitely many $a_e$ are nonzero in the local contractible curve setting. Since $K_e(1)=e^2$, specializing at $\lambda=0$ gives
\begin{equation}\label{eq:Rclassical}
\boxed{
R_g(q)=-\sum_{e\ge1}e^2a_e\Li_1(q^{ge}).}
\end{equation}

\begin{lemma}\label{lem:divisor}
For every $N\ge1$,
\begin{equation}\label{eq:divisor}
\boxed{
[q^{gN}]R_g(q)=-\frac1N\sum_{e\mid N}e^3a_e.}
\end{equation}
\end{lemma}

\begin{proof}
The term $-e^2a_e\Li_1(q^{ge})$ contributes to $q^{gN}$ precisely when $N=ek$, equivalently $e\mid N$. Its coefficient is then
\[
-e^2a_e\frac1k=-\frac{e^3}{N}a_e.
\]
Summing over $e\mid N$ gives \eqref{eq:divisor}.
\end{proof}

Define
\begin{equation}\label{eq:bcdef}
c_N=[q^{gN}]R_g(q),\qquad b_N=-Nc_N.
\end{equation}
Then
\begin{equation}\label{eq:bdivisor}
\boxed{b_N=\sum_{e\mid N}e^3a_e.}
\end{equation}

\subsection{M\"obius reconstruction}

\begin{theorem}[Arithmetic reconstruction]\label{thm:reconstruction}
Fix the exceptional ray and its lattice normalization $g$. Then the classical forcing series $R_g(q)$ uniquely determines the local genus zero GV spectrum $\{n_{gN}\}_{N\ge1}$. More precisely,
\begin{equation}\label{eq:moebius-b}
\boxed{
N^3a_N=\sum_{r\mid N}\mu(r)b_{N/r},}
\end{equation}
and hence
\begin{equation}\label{eq:moebius-main}
\boxed{
n_{gN}=a_N=-\frac1{N^2}\sum_{r\mid N}\frac{\mu(r)}r
[q^{gN/r}]R_g(q).}
\end{equation}
\end{theorem}

\begin{proof}
Equation \eqref{eq:bdivisor} has the form
\[
b_N=\sum_{e\mid N}f(e),\qquad f(e)=e^3a_e.
\]
M\"obius inversion gives
\[
f(N)=\sum_{r\mid N}\mu(r)b_{N/r},
\]
which is \eqref{eq:moebius-b}. Since
\[
b_{N/r}=-\frac Nr[q^{gN/r}]R_g(q),
\]
dividing by $N^3$ yields \eqref{eq:moebius-main}.
\end{proof}

\begin{corollary}\label{cor:R-spectrum}
For fixed $g$, the data $R_g(q)$ and $\{n_{gN}\}_{N\ge1}$ mutually determine one another.
\end{corollary}

\begin{remark}\label{rem:g-normalization}
The normalization $g$ is part of the data. The theorem does not assert that an abstract formal series canonically determines the primitive lattice on which it should be interpreted. Once the exceptional ray and its lattice normalization are fixed, however, the forcing coefficients determine all local GV multiplicities on that ray.
\end{remark}

\subsection{Reconstruction from the full quantum forcing}

The full forcing has a divisor triangular coefficient structure.

\begin{proposition}\label{prop:quantum-triangle}
For every $N\ge1$,
\begin{equation}\label{eq:quantum-coeff}
\boxed{
[q^{gN}]\R_g(\lambda,q)
=-\sum_{e\mid N}a_e\frac eN
K_e\!\left(e^{i(N/e)\lambda}\right).}
\end{equation}
\end{proposition}

\begin{proof}
The $e$-sector contributes to $q^{gN}$ only when $N=ek$. The coefficient of $q^{gek}$ in $\Li_1(q^{ge})$ is $1/k=e/N$. On this mode $T_{h/(ge)}$ acts by $e^{ik\lambda}=e^{i(N/e)\lambda}$, giving \eqref{eq:quantum-coeff}.
\end{proof}

The diagonal divisor $e=N$ contributes $-a_NK_N(e^{i\lambda})$. Since
\begin{equation}\label{eq:Kunit}
K_N(e^{i\lambda})=N^2+O(\lambda^2),
\end{equation}
it is a unit in $\CC[[\lambda]]$.

\begin{corollary}[Quantum triangular reconstruction]\label{cor:quantum-reconstruction}
For every $N\ge1$,
\begin{equation}\label{eq:quantum-recursion}
\boxed{
a_N=-\frac{
[q^{gN}]\R_g+
\displaystyle\sum_{\substack{e\mid N\\ e<N}}
a_e\frac eN K_e(e^{i(N/e)\lambda})
}{K_N(e^{i\lambda})}.}
\end{equation}
Consequently the full forcing term $\R_g(\lambda,q)$ determines the entire local genus zero GV spectrum.
\end{corollary}

\begin{proof}
Separate $e=N$ in \eqref{eq:quantum-coeff} and solve for $a_N$. Induction on $N$ completes the argument.
\end{proof}

\subsection{Information equivalence}

\begin{theorem}[Information equivalence]\label{thm:information}
Fix the exceptional ray and its lattice normalization $g$. Then the following three data sets mutually determine one another:
\begin{equation}\label{eq:information-equivalence}
\boxed{
\R_g(\lambda,q)
\Longleftrightarrow
R_g(q)
\Longleftrightarrow
\{n_{gN}\}_{N\ge1}.}
\end{equation}
\end{theorem}

\begin{proof}
The map $\R_g\mapsto R_g$ is specialization at $\lambda=0$. The map $R_g\mapsto\{n_{gN}\}$ is \cref{thm:reconstruction}. Conversely the spectrum determines the full forcing by
\begin{equation}\label{eq:forcing-from-spectrum}
\R_g(\lambda,q)
=-\sum_{N\ge1}n_{gN}K_N(T_{h/(gN)})\Li_1(q^{gN}),
\end{equation}
which is \eqref{eq:Rquantum} with $e=N$.
\end{proof}

\begin{corollary}\label{cor:equality}
Let $C$ and $C'$ be two local flopping geometries whose exceptional rays are identified with the same lattice normalization $g$. Then $\R_g=\R'_g$ if and only if $R_g=R'_g$, if and only if
\[
n_{gN}=n'_{gN}\qquad\text{for every }N\ge1.
\]
\end{corollary}

\begin{remark}
Theorem~\ref{thm:information} does not say that the $\lambda$-dependent forcing has no additional structure. It records the complete shift profile $K_N(e^{ik\lambda})$ associated with each multiple cover sector. The point is more specific: for recovering the local genus zero GV multiplicities, this quantum shift structure is redundant. Thus
\[
\lambda\to0\quad\text{forgets the shift profile but not the GV spectrum.}
\]
\end{remark}

\subsection{Contraction algebra widths from the classical forcing}\label{subsec:contraction}

We now record a geometric consequence of the reconstruction package. Consider a one-curve flopping contraction in the primitive curve class normalization, so that the local spectrum is $\{n_d\}_{1\le d\le\ell}$ and $g=1$. Define
\begin{equation}\label{eq:Eforcing}
\mathscr E_C(q):=\exp(R_1(q)).
\end{equation}
Since $\Li_1(z)=-\log(1-z)$, equation \eqref{eq:Rclassical} gives
\begin{equation}\label{eq:Eforcing-product}
\boxed{
\mathscr E_C(q)=\prod_{d=1}^{\ell}(1-q^d)^{d^2n_d}.}
\end{equation}
Only finitely many factors are nontrivial.

Recall that the contraction algebra $A_{\mathrm{con}}$ of a flopping curve is the finite dimensional algebra representing its universal noncommutative deformation theory \cite{DonovanWemyss2016}. Toda proved that its noncommutative and commutative widths are determined by the genus zero GV invariants by
\begin{equation}\label{eq:Toda-widths}
\dim_{\CC}A_{\mathrm{con}}=\sum_{d=1}^{\ell}d^2n_d,
\qquad
\dim_{\CC}A_{\mathrm{con}}^{\mathrm{ab}}=n_1
\end{equation}
\cite{Toda2015}.

\begin{proposition}[Contraction algebra bridge]\label{prop:contraction-bridge}
The classical difference forcing determines both widths of the contraction algebra. More precisely,
\begin{equation}\label{eq:width-from-forcing}
\boxed{
\ord_{q=1}\mathscr E_C(q)=\dim_{\CC}A_{\mathrm{con}},
\qquad
-[q]R_1(q)=\dim_{\CC}A_{\mathrm{con}}^{\mathrm{ab}}.}
\end{equation}
\end{proposition}

\begin{proof}
For each $d\ge1$, the factor $1-q^d$ has a simple zero at $q=1$. Hence \eqref{eq:Eforcing-product} gives
\[
\ord_{q=1}\mathscr E_C(q)=\sum_{d=1}^{\ell}d^2n_d,
\]
which is the first identity by \eqref{eq:Toda-widths}. For the second, only the $d=1$ sector contributes to the coefficient of $q$ in
\[
R_1(q)=-\sum_{d=1}^{\ell}d^2n_d\Li_1(q^d),
\]
so $-[q]R_1(q)=n_1$. Apply \eqref{eq:Toda-widths} again.
\end{proof}

\begin{corollary}[A forcing criterion for commutativity]\label{cor:commutative-forcing}
The contraction algebra is commutative if and only if
\begin{equation}\label{eq:commutative-criterion}
\boxed{
\ord_{q=1}\mathscr E_C(q)=-[q]R_1(q).}
\end{equation}
\end{corollary}

\begin{proof}
The natural map $A_{\mathrm{con}}\to A_{\mathrm{con}}^{\mathrm{ab}}$ is surjective, and equality of the two dimensions is equivalent to vanishing of the commutator ideal. The claim follows from Proposition~\ref{prop:contraction-bridge}.
\end{proof}

Thus the classical forcing is not only equivalent to the local GV spectrum; two standard numerical invariants of noncommutative flop geometry can be read directly from its behavior at $q=0$ and $q=1$.

\subsection{The fine common shift lattice}

Assume from now on that $S\subset\ZZ_{>0}$ is finite and nonempty. Set
\begin{equation}\label{eq:Ldef}
L=\lcm(S),\qquad Z=T_{h/L}.
\end{equation}
For $d\in S$ define $m_d=L/d$, so $T_{h/d}=Z^{m_d}$. On $q^{dk}$, $Z$ has eigenvalue
\[
x=e^{idk\lambda/L}=e^{ik\lambda/m_d},
\]
so $x^{m_d}=e^{ik\lambda}$. Hence
\begin{equation}\label{eq:denom-x}
\boxed{
(2\sin(k\lambda/2))^2=-\frac{(x^{m_d}-1)^2}{x^{m_d}}.}
\end{equation}
The Laurent monomial $-x^{-m_d}$ is a unit; the essential denominator clearing factor is therefore $(x^{m_d}-1)^2$.

\begin{definition}\label{def:denominator-clearing}
A Laurent polynomial $P(Z)\in\CC[Z^{\pm1}]$ is called a \emph{denominator clearing operator for $S$} if
\begin{equation}\label{eq:denom-def}
\boxed{
(Z^{L/d}-1)^2\mid P(Z)\qquad\text{for every }d\in S.}
\end{equation}
Two such operators differing by a Laurent unit $cZ^r$, $c\in\CC^\times$, $r\in\ZZ$, are regarded as equivalent. Minimality is with respect to divisibility in $\CC[Z^{\pm1}]$ modulo these units.
\end{definition}

\begin{lemma}\label{lem:denom-equiv}
The divisibility condition in Definition~\ref{def:denominator-clearing} is precisely the polynomial condition induced by cancellation of the sine square multiple cover denominator on every sector mode of degree $d\in S$.
\end{lemma}

\begin{proof}
Fix $d\in S$ and a mode $q^{dk}$. By \eqref{eq:denom-x}, the sine square denominator is, up to the invertible monomial $-x^{-L/d}$, the square $(x^{L/d}-1)^2$. Thus, mode by mode, cancellation by the eigenvalue of $P(Z)$ requires divisibility by this polynomial. Passing from the eigenvalue variable $x$ to the indeterminate $Z$ yields exactly the algebraic condition \eqref{eq:denom-def}.
\end{proof}

\subsection{Cyclotomic minimality}

\begin{theorem}[Fine lattice cyclotomic minimality]\label{thm:cyclotomic}
Let $S\subset\ZZ_{>0}$ be finite and nonempty and let $L=\lcm(S)$. Up to multiplication by a Laurent unit, the unique minimal denominator clearing operator is
\begin{equation}\label{eq:Pmin}
\boxed{
P_S^{\min}(Z)=\lcm_{d\in S}(Z^{L/d}-1)^2.}
\end{equation}
Define
\begin{equation}\label{eq:Cset}
\mathscr C(S)=\bigcup_{d\in S}\left\{r\ge1:r\mid\frac Ld\right\}.
\end{equation}
Then, up to a Laurent unit,
\begin{equation}\label{eq:Pcyclo}
\boxed{
P_S^{\min}(Z)=\prod_{r\in\mathscr C(S)}\Phi_r(Z)^2.}
\end{equation}
\end{theorem}

\begin{proof}
Any denominator clearing operator must be divisible by $(Z^{L/d}-1)^2$ for every $d\in S$, hence by their least common multiple. Conversely that least common multiple satisfies every required divisibility condition and is therefore minimal.

For the cyclotomic factorization, use
\[
Z^m-1=\prod_{r\mid m}\Phi_r(Z).
\]
Thus
\[
(Z^{L/d}-1)^2=\prod_{r\mid L/d}\Phi_r(Z)^2.
\]
Taking the least common multiple over $d\in S$ retains precisely those factors with $r\mid L/d$ for at least one $d\in S$, namely $r\in\mathscr C(S)$, each with exponent two.
\end{proof}

\begin{corollary}[Primitive degree simplification]\label{cor:primitive}
If $1\in S$, then, up to a Laurent unit,
\begin{equation}\label{eq:Pprimitive}
\boxed{P_S^{\min}(Z)=(Z^L-1)^2.}
\end{equation}
\end{corollary}

\begin{proof}
The factor corresponding to $d=1$ is $(Z^L-1)^2$. For every $d\in S$, $L/d\mid L$, hence $Z^{L/d}-1\mid Z^L-1$.
\end{proof}

Since $Z^L=T_h$,
\begin{equation}\label{eq:symmetric-rep}
Z^{-L}(Z^L-1)^2=Z^L-2+Z^{-L}=T_h-2+T_{-h}.
\end{equation}

\begin{corollary}\label{cor:skeleton}
If $1\in S$, the minimal denominator clearing operator has the symmetric central difference representative
\[
\boxed{\Delta_1=T_h-2+T_{-h}.}
\]
Thus additional nonzero degrees need not enlarge the minimal symmetric difference skeleton.
\end{corollary}

\subsection{Gcd lattice versus fine lattice}

The two constructions answer different questions. The gcd lattice, governed by $g=\gcd S$, produces a single central difference equation. The fine lattice, governed by $L=\lcm(S)$, produces a minimal Laurent polynomial denominator clearing operator. For example, if $S=\{1,2\}$, then $g=1$ and $L=2$. The gcd lattice equation uses $\Delta_1$, whereas on the fine lattice $Z=T_{h/2}$,
\[
P_S^{\min}(Z)=(Z^2-1)^2.
\]
These differ only by the Laurent unit $Z^2$ because
\[
Z^{-2}(Z^2-1)^2=Z^2-2+Z^{-2}=\Delta_1.
\]
Thus the two viewpoints are compatible but conceptually distinct.

\subsection{The information hierarchy}

For a fixed exceptional ray, the preceding results yield
\begin{equation}\label{eq:hierarchy3}
\boxed{
\R_g\Longleftrightarrow R_g\Longleftrightarrow\{n_d\}
\Longrightarrow S\Longrightarrow P_S^{\min}.}
\end{equation}
The first three objects mutually determine one another by \cref{thm:information}. Passing from the spectrum to its support forgets multiplicities; passing from $S$ to $P_S^{\min}$ may forget further arithmetic information. The next section exhibits explicit flopping geometries realizing these losses. Moreover, even the full forcing need not determine the analytic type of the contraction.

\section{Explicit flops and information loss}\label{sec:examples}

We now apply the general results to explicit threefold flops. Throughout the examples below the degree one sector occurs, so $g=1$ and the common central difference operator is
\[
\Delta_1=T_h-2+T_{-h}.
\]

\subsection{The resolved conifold}

For the resolved conifold, $S=\{1\}$ and $n_1=1$. Hence
\[
\F_{\con}^{\loc}=\Phi_1,
\]
and
\begin{equation}\label{eq:conifold-diff}
\Delta_1\F_{\con}^{\loc}=D^2\Li_3(q)=\log(1-q).
\end{equation}
Writing $Z_{\con}^{\loc}=\exp(\F_{\con}^{\loc})$ gives
\begin{equation}\label{eq:conifold-Z}
\boxed{
\frac{Z_{\con}^{\loc}(t+h)Z_{\con}^{\loc}(t-h)}{Z_{\con}^{\loc}(t)^2}=1-q.}
\end{equation}

\subsection{The standard Laufer flop}

Consider an isolated length two example with local genus zero GV spectrum
\begin{equation}\label{eq:laufer-spectrum}
\boxed{(n_1,n_2)=(5,1).}
\end{equation}
This is the genus zero GV spectrum of the standard Laufer flop and also occurs in the Brown-Wemyss comparison; see \cite{BrownWemyss2018,Toda2015,CollinucciEtAl2022}. The local contribution is
\begin{equation}\label{eq:laufer-F}
\boxed{\F_L^{\loc}=5\Phi_1+\Phi_2.}
\end{equation}
Since $K_2(z)=z+2+z^{-1}$, \cref{thm:common} gives
\begin{equation}\label{eq:laufer-D2}
\boxed{
\Delta_1\F_L^{\loc}=D^2\left[
5\Li_3(q)+\frac14(T_{h/2}+2+T_{-h/2})\Li_3(q^2)
\right].}
\end{equation}
Using \eqref{eq:D2Li3},
\[
\begin{aligned}
\Delta_1\F_L^{\loc}
={}&5\log(1-q)
+\log(1-q^2e^{i\lambda})+2\log(1-q^2)\\
&+\log(1-q^2e^{-i\lambda}).
\end{aligned}
\]
Therefore
\begin{equation}\label{eq:laufer-log}
\boxed{
\Delta_1\F_L^{\loc}
=\log\!\left[(1-q)^5(1-q^2)^2(1-q^2e^{i\lambda})(1-q^2e^{-i\lambda})\right].}
\end{equation}
Equivalently,
\begin{equation}\label{eq:laufer-cos}
\Delta_1\F_L^{\loc}
=\log\!\left[(1-q)^5(1-q^2)^2(1-2q^2\cos\lambda+q^4)\right].
\end{equation}
If $Z_L^{\loc}=\exp(\F_L^{\loc})$, then
\begin{equation}\label{eq:laufer-Z}
\boxed{
\frac{Z_L^{\loc}(t+h)Z_L^{\loc}(t-h)}{Z_L^{\loc}(t)^2}
=(1-q)^5(1-q^2)^2(1-q^2e^{i\lambda})(1-q^2e^{-i\lambda}).}
\end{equation}

The classical forcing is
\begin{equation}\label{eq:laufer-R}
\boxed{R_L(q)=-5\Li_1(q)-4\Li_1(q^2).}
\end{equation}
The first coefficients are
\[
[q]R_L=-5,\qquad [q^2]R_L=-\frac{13}{2}.
\]
Thus $b_1=5$ and $b_2=13$. Equation \eqref{eq:moebius-b} yields
\[
n_1=5,\qquad 2^3n_2=b_2-b_1=8,
\]
recovering $n_2=1$.

Finally, $S_L=\{1,2\}$ and $L=2$, so
\begin{equation}\label{eq:laufer-P}
P_{S_L}^{\min}(Z)=(Z^2-1)^2,\qquad Z=T_{h/2},
\end{equation}
up to a Laurent unit. Its symmetric representative is again $\Delta_1$.

\subsection{The length three E6 model}

For the explicit length three $E_6$ model of \cite{CollinucciEtAl2022}, the local genus zero GV spectrum is
\begin{equation}\label{eq:E6-spectrum}
\boxed{(n_1,n_2,n_3)=(6,3,1).}
\end{equation}
Thus
\begin{equation}\label{eq:E6-F}
\boxed{\F_{E_6}^{\loc}=6\Phi_1+3\Phi_2+\Phi_3.}
\end{equation}
Using
\[
K_3(z)=z^2+2z+3+2z^{-1}+z^{-2},
\]
we obtain
\begin{equation}\label{eq:E6-D2}
\boxed{
\begin{aligned}
\Delta_1\F_{E_6}^{\loc}
=D^2\Big[&6\Li_3(q)+\frac34K_2(T_{h/2})\Li_3(q^2)\\
&+\frac19K_3(T_{h/3})\Li_3(q^3)\Big].
\end{aligned}}
\end{equation}
Expanding the finite shifts gives
\begin{equation}\label{eq:E6-log}
\begin{aligned}
\Delta_1\F_{E_6}^{\loc}
={}&6\log(1-q)\\
&+3\log(1-q^2e^{i\lambda})+6\log(1-q^2)+3\log(1-q^2e^{-i\lambda})\\
&+\log(1-q^3e^{2i\lambda})+2\log(1-q^3e^{i\lambda})+3\log(1-q^3)\\
&+2\log(1-q^3e^{-i\lambda})+\log(1-q^3e^{-2i\lambda}).
\end{aligned}
\end{equation}
Therefore
\begin{equation}\label{eq:E6-Z}
\boxed{
\begin{aligned}
\frac{Z_{E_6}^{\loc}(t+h)Z_{E_6}^{\loc}(t-h)}{Z_{E_6}^{\loc}(t)^2}
={}&(1-q)^6\\
&\times(1-q^2e^{i\lambda})^3(1-q^2)^6(1-q^2e^{-i\lambda})^3\\
&\times(1-q^3e^{2i\lambda})(1-q^3e^{i\lambda})^2(1-q^3)^3\\
&\times(1-q^3e^{-i\lambda})^2(1-q^3e^{-2i\lambda}).
\end{aligned}}
\end{equation}
The classical forcing is
\begin{equation}\label{eq:E6-R}
\boxed{R_{E_6}(q)=-6\Li_1(q)-12\Li_1(q^2)-9\Li_1(q^3).}
\end{equation}
Its first coefficients are
\[
[q]R_{E_6}=-6,\qquad [q^2]R_{E_6}=-15,\qquad [q^3]R_{E_6}=-11.
\]
Thus $b_1=6$, $b_2=30$, $b_3=33$, and M\"obius inversion gives
\[
2^3n_2=30-6=24,\qquad 3^3n_3=33-6=27,
\]
recovering \eqref{eq:E6-spectrum}.

For $S_{E_6}=\{1,2,3\}$, $L=6$, and
\begin{equation}\label{eq:E6-P}
\boxed{
P_{E_6}^{\min}(Z)=(Z^6-1)^2
=\Phi_1(Z)^2\Phi_2(Z)^2\Phi_3(Z)^2\Phi_6(Z)^2,}
\end{equation}
up to a Laurent unit.

\subsection{Higher length E7 and E8 models}

The gauge theoretic constructions of \cite{CollinucciEtAl2022} produce explicit simple flops of lengths $1,\dots,6$ and allow the genus zero GV multiplicities to be read off. For the higher length representatives relevant here,
\begin{equation}\label{eq:higher-table}
\begin{array}{c|c|c}
\text{length} & \text{model} & \text{GV spectrum}\\
\hline
3 & E_6 &(6,3,1)\\
4 & E_7 &(6,5,2,1)\\
5 & E_8 &(8,6,4,2,1)\\
6 & E_8 &(6,6,4,3,2,1).
\end{array}
\end{equation}
These are particular explicit representatives; the length itself should not be interpreted as uniquely determining the GV multiplicities.

For every model in \eqref{eq:higher-table},
\begin{equation}\label{eq:higher-master}
\boxed{
\Delta_1\F^{\loc}
=D^2\sum_{d=1}^{\ell}\frac{n_d}{d^2}K_d(T_{h/d})\Li_3(q^d).}
\end{equation}
The higher length geometries successively realize the kernels
\[
K_3,\quad K_4,\quad K_5,\quad K_6.
\]
For reference,
\begin{equation}\label{eq:K4}
K_4(z)=z^3+2z^2+3z+4+3z^{-1}+2z^{-2}+z^{-3},
\end{equation}
\begin{equation}\label{eq:K5}
\begin{aligned}
K_5(z)={}&z^4+2z^3+3z^2+4z+5\\
&+4z^{-1}+3z^{-2}+2z^{-3}+z^{-4},
\end{aligned}
\end{equation}
and
\begin{equation}\label{eq:K6}
\begin{aligned}
K_6(z)={}&z^5+2z^4+3z^3+4z^2+5z+6\\
&+5z^{-1}+4z^{-2}+3z^{-3}+2z^{-4}+z^{-5}.
\end{aligned}
\end{equation}
The length six $E_8$ example therefore realizes the largest kernel among these simple flop models.

For the $E_7$ spectrum $(6,5,2,1)$,
\begin{equation}\label{eq:E7-D2}
\boxed{
\begin{aligned}
\Delta_1\F_{E_7}^{\loc}=D^2\Big[&6\Li_3(q)+\frac54K_2(T_{h/2})\Li_3(q^2)\\
&+\frac29K_3(T_{h/3})\Li_3(q^3)+\frac1{16}K_4(T_{h/4})\Li_3(q^4)\Big],
\end{aligned}}
\end{equation}
with classical forcing
\begin{equation}\label{eq:E7-R}
\boxed{R_{E_7}(q)=-6\Li_1(q)-20\Li_1(q^2)-18\Li_1(q^3)-16\Li_1(q^4).}
\end{equation}
For the length five $E_8$ spectrum $(8,6,4,2,1)$,
\begin{equation}\label{eq:E85-D2}
\boxed{
\begin{aligned}
\Delta_1\F_{E_8^{(5)}}^{\loc}=D^2\Big[&8\Li_3(q)+\frac32K_2(T_{h/2})\Li_3(q^2)\\
&+\frac49K_3(T_{h/3})\Li_3(q^3)+\frac18K_4(T_{h/4})\Li_3(q^4)\\
&+\frac1{25}K_5(T_{h/5})\Li_3(q^5)\Big],
\end{aligned}}
\end{equation}
and
\begin{equation}\label{eq:E85-R}
\boxed{
\begin{aligned}
R_{E_8^{(5)}}(q)={}&-8\Li_1(q)-24\Li_1(q^2)-36\Li_1(q^3)\\
&-32\Li_1(q^4)-25\Li_1(q^5).
\end{aligned}}
\end{equation}
For the length six $E_8$ spectrum $(6,6,4,3,2,1)$,
\begin{equation}\label{eq:E86-D2}
\boxed{
\begin{aligned}
\Delta_1\F_{E_8^{(6)}}^{\loc}=D^2\Big[&6\Li_3(q)+\frac32K_2(T_{h/2})\Li_3(q^2)\\
&+\frac49K_3(T_{h/3})\Li_3(q^3)+\frac3{16}K_4(T_{h/4})\Li_3(q^4)\\
&+\frac2{25}K_5(T_{h/5})\Li_3(q^5)+\frac1{36}K_6(T_{h/6})\Li_3(q^6)\Big],
\end{aligned}}
\end{equation}
and
\begin{equation}\label{eq:E86-R}
\boxed{
\begin{aligned}
R_{E_8^{(6)}}(q)={}&-6\Li_1(q)-24\Li_1(q^2)-36\Li_1(q^3)\\
&-48\Li_1(q^4)-50\Li_1(q^5)-36\Li_1(q^6).
\end{aligned}}
\end{equation}

\subsection{Same length, different forcing}

The GV spectrum is finer than Koll\'ar's length. The length two family considered in \cite{CollinucciEtAl2022} exhibits a generic spectrum $(4,1)$ and a special spectrum $(5,1)$. Their classical forcing terms are
\begin{equation}\label{eq:length2-generic}
R_{(4,1)}=-4\Li_1(q)-4\Li_1(q^2)
\end{equation}
and
\begin{equation}\label{eq:length2-special}
R_{(5,1)}=-5\Li_1(q)-4\Li_1(q^2),
\end{equation}
respectively. Hence
\begin{equation}\label{eq:length-not-force}
\boxed{
\text{same Koll\'ar length}\not\Longrightarrow\text{same local difference forcing}.}
\end{equation}
Thus the forcing detects enumerative information invisible to the length alone.

\subsection{Different supports with the same cyclotomic skeleton}

The length five and length six $E_8$ models provide a complementary phenomenon. Their supports are
\[
S_5=\{1,2,3,4,5\},\qquad S_6=\{1,2,3,4,5,6\},
\]
so $S_5\ne S_6$. However,
\begin{equation}\label{eq:lcm60}
\lcm(S_5)=60=\lcm(S_6).
\end{equation}
Let $Z=T_{h/60}$. Since $1\in S_5\cap S_6$, \cref{cor:primitive} gives
\begin{equation}\label{eq:E8sameP}
\boxed{
P_{S_5}^{\min}(Z)=P_{S_6}^{\min}(Z)=(Z^{60}-1)^2}
\end{equation}
up to Laurent units.

\begin{corollary}\label{cor:same-skeleton}
There exist explicit simple flop models with different GV supports but identical fine lattice minimal denominator clearing operators. In particular,
\[
\boxed{P_S^{\min}\not\Longrightarrow S.}
\]
\end{corollary}

The forcing terms nevertheless distinguish these examples, since \eqref{eq:E85-R} and \eqref{eq:E86-R} are different. Thus the passage from forcing data to the cyclotomic skeleton is genuinely information losing in geometric examples.

\subsection{Identical forcing does not determine the flop}

Brown and Wemyss constructed two threefold flops, each with a single flopping curve, which are neither algebraically nor analytically isomorphic but have identical curve counting GV invariants \cite{BrownWemyss2018}. After identifying the exceptional curve classes, \cref{thm:information} therefore gives identical full and classical forcing terms.

\begin{corollary}\label{cor:not-analytic}
The local Gromov-Witten difference forcing does not determine the analytic type of a threefold flop. More precisely, there exist analytically non-isomorphic flops with identical
\[
\boxed{
\R_g,\qquad R_g,\qquad\{n_d\},\qquad S,\qquad P_S^{\min}.}
\]
\end{corollary}

Thus the local curve counting difference data considered here remain strictly coarser than the analytic geometry of the contraction. This is consistent with the fact that contraction algebraic data can distinguish flops that curve counting invariants do not \cite{DonovanWemyss2016,HuaToda2017,Toda2015}.

\subsection{Three levels of information loss}

Combining the preceding examples yields
\begin{equation}\label{eq:three-losses}
\boxed{
\begin{array}{rcl}
\text{same Koll\'ar length}&\not\Longrightarrow&\text{same forcing},\\[1mm]
\text{same cyclotomic skeleton}&\not\Longrightarrow&\text{same GV support},\\[1mm]
\text{same forcing}&\not\Longrightarrow&\text{same analytic flop}.
\end{array}}
\end{equation}
For the fixed exceptional ray,
\[
\R_g\Longleftrightarrow R_g\Longleftrightarrow\{n_d\}
\]
contains precisely the same genus zero enumerative information. The support $S$ is coarser, and the cyclotomic operator $P_S^{\min}$ is coarser still. At the same time, the entire local curve counting package is not fine enough to recover the analytic type. Thus
\begin{equation}\label{eq:final-hierarchy}
\boxed{
\text{analytic flop geometry}
\longrightarrow
\bigl(\R_g\Longleftrightarrow R_g\Longleftrightarrow\{n_d\}\bigr)
\longrightarrow S\longrightarrow P_S^{\min},}
\end{equation}
with the failures of reverse implication realized by the examples above.

\section{Discussion and outlook}\label{sec:discussion}

The results above separate three roles of local Gromov-Witten difference equations. First, the common gcd lattice equation packages all nonzero degree sectors into one functional relation. Second, its forcing term is a complete reconstruction device for the genus zero local GV spectrum: passing to $\lambda=0$ removes the quantum shift profile but does not lose the curve counting multiplicities. Third, after exponentiation, the classical forcing has a direct interpretation in contraction algebra geometry through Proposition~\ref{prop:contraction-bridge}.

This reconstruction statement should be distinguished from the stronger problem of recovering the analytic type of the contraction. The Brown-Wemyss examples show that the complete local GV package, and hence the full forcing considered here, remains too coarse for that purpose. Conversely, the cyclotomic denominator clearing operator is intentionally coarser than the GV spectrum and records only part of the arithmetic support. The resulting hierarchy makes precise where information is preserved and where it is lost.

The recent resurgent analysis of the resolved conifold by Alim and Tischler \cite{AlimTischler2026} provides a complementary perspective. They construct a global analytic, nonperturbative completion in the topological string coupling whose weak coupling asymptotics recover the closed Gromov-Witten expansion, while its strong coupling expansion in dual variables contains a Nekrasov-Shatashvili contribution with a precise interpretation in terms of relative Gromov-Witten invariants of \(\mathbb P^1\) with maximal tangency at \(\infty\). Thus a single analytic object interpolates between two a priori different curve counting theories. It would be interesting to understand whether the multi-sector difference forcing studied here admits an analogous resurgent or nonperturbative completion, and, if so, whether its strong coupling sectors have a natural relative, logarithmic, or open Gromov-Witten interpretation.

Two further extensions are suggested directly by the reconstruction mechanism. First, the genus zero hypothesis is essential to the information equivalence theorem in its present form. If
\[
\Phi_{r,d}(\lambda,t):=\sum_{k\ge1}\frac{q^{dk}}{k}
\left(2\sin\frac{k\lambda}{2}\right)^{2r-2},\qquad r\ge0,
\]
denotes the contribution of a fixed genus $r$ BPS sector, then the same elementary shift computation gives
\begin{equation}\label{eq:genus-raising-outlook}
\Delta_d\Phi_{r,d}=-\Phi_{r+1,d}.
\end{equation}
Thus central difference acts as a genus-raising operator on the GV multiple cover factors. In particular, unlike the genus zero situation studied here, the nontrivial $\lambda$-dependence of the forcing is expected to carry genuinely new higher genus BPS information rather than being redundant after specialization at $\lambda=0$. This suggests an all genus reconstruction problem combining divisor inversion in the curve degree with a triangular reconstruction in the genus variable.

Second, an orbifold analogue should naturally be sector-valued rather than scalar. Bryan and Pietromonaco recently defined integral orbifold Gopakumar-Vafa invariants for local Calabi-Yau orbifolds with transverse $A_N$ singularities and proved a finiteness property \cite{BryanPietromonaco2026}. In such a setting one expects the divisor structure of multiple covers to interact with isotropy characters and twisted sectors, suggesting a reconstruction theory combining divisor inversion with finite Fourier data. We do not develop either extension here. 

In conclusion, the scope of the present paper is deliberately focused on the scalar genus zero theory associated with a single exceptional ray. Our aim is to understand the reconstruction of local Gopakumar-Vafa data from difference forcing, the associated denominator clearing problem, its relation to contraction algebra widths, and the resulting hierarchy of information loss. Extensions beyond this setting involve additional geometric and analytic structures and will not be pursued here.

\section*{Acknowledgement} Xiaobin Li is supported by NSFC grant No.11501470, No.11426187, No.11791240561, and partially supported by NSFC grant No.11671328, Chengdu Science and Technology Program with Grant No.2025-YF09-00007-SN and the Fundamental Research Funds for the Central Universities 2682021ZTPY043, 2682025ZTPY001, 2682025ZTPY057, 2682025ZTO002. Xiaobin Li thanks Sung-Soo Kim, Kimyeong Lee, Futoshi Yagi, Satoshi Nawata and Rui-Dong Zhu for useful discussions and collaborations in physics. In particular, Xiaobin Li is grateful to Arpan Saha for his inspiring online seminars delivered during the pandemic, and to Doan Nhat Minh and Van Nguyen for the enjoyable and fruitful collaboration. Xiaobin Li would also like to thank Bohui Chen, An-min Li, Guosong Zhao for their constant support and all friends met at various conferences. Finally, Xiaobin Li expresses special gratitude to the Mainz Institute for Theoretical Physics (MITP) of the Cluster of Excellence PRISMA\(^{+}\) (Project ID 390831469), for its hospitality and support.

\end{document}